\documentclass[12pt, reqno]{amsart}

\usepackage{amsmath, amsthm, amscd, amsfonts, amssymb, graphicx, color}
\usepackage[bookmarksnumbered, colorlinks, plainpages]{hyperref}
\usepackage{mathtools}
\usepackage{enumerate}
\usepackage{amscd} 
\usepackage{xy}
\xyoption{all}

\newtheorem{theorem}{Theorem}[section]
\newtheorem{lemma}[theorem]{Lemma}

\newtheorem{corollary}[theorem]{Corollary}
\theoremstyle{definition}

\newtheorem{remark}[theorem]{Remark}

\begin{document}
\setcounter{page}{1}

\title[ ]{Some Analytical Properties of the Hyperbolic Sine Integral}

\author[K. Nantomah]{Kwara Nantomah}

\address{Department of Mathematics, School of Mathematical Sciences, C. K. Tedam University of Technology and Applied Sciences, P. O. Box 24, Navrongo, Upper-East Region, Ghana. }
\email{\textcolor[rgb]{0.00,0.00,0.84}{ knantomah@cktutas.edu.gh}}





\begin{abstract}
By using some tools of analysis, we establish some analytical properties such as monotonicity and inequalities involving the hyperbolic sine integral function. As applications of some of the established properties, we obtain some rational bounds for the hyperbolic tangent function. 
\end{abstract} \maketitle


\section{Introduction}

The cardinal hyperbolic sine function which is also known as sinhc function or hyperbolic sinc function is defined for $z\in(-\infty,\infty)$ as  \cite{Sanchez-2012-TCMJ}
\begin{equation}\label{eqn:Hyper-Sinc-Funct}
\mathrm{sinhc}(z)=
 \begin{cases} 
      \frac{\sinh(z)}{z}, & z\neq 0 \\
       1, & z=0 .
   \end{cases}     
\end{equation}
It has been very useful in various areas of mathematics, physics and engineering. For example, it has been demonstrated that the function exhibits a clear geometric interpretation as the ratio between length and chord of a symmetric catenary segment \cite{Merca-2016-JNT}, \cite{Sanchez-2012-TCMJ}. Due to its usefulness, it has been investigated by several researchers and many remarkable inequalities have been established concerning the function. For further information and recent developments on such inequalities, one may consult the works \cite{Bagul-2017-JMI}, \cite{Bagul-Chesneau-2019-IJOPCM}, \cite{Bagul-Chesneau-2019-CUBO}, \cite{Bagul-Etal-2021-Axioms}, \cite{Bagul-Etal-2022-AMS}, \cite{Li-Miao-Guo-2022-Axioms}, \cite{Nantomah-Prempeh-2020-MJPAA}, \cite{Sandor-2017-NNTDM}, \cite{Zhu-2010-JIA} and the references therein.

Closely related to the sinch function is the hyperbolic sine integral function which is defined for $z\in(-\infty,\infty)$ as \cite[p. 231]{Abramowitz-Stegun-1965-DP} 
\begin{equation}\label{eqn:HyperSin-Integral-IntRep-1}
\mathrm{Shi}(z)=\int_{0}^{z} \frac{\sinh(t)}{t}\,dt.
\end{equation}
In \cite{Nijimbere-2018-UMJ}, the author considered three representations in terms of the hypergeometric function $_2F_3$ for a certain indefinite hyperbolic sine integral. A review of the literature reveals that, unlike the cardinal hyperbolic sine function which is well researched in terms of its inequalities or bounds, the hyperbolic sine integral is yet to receive a similar attention.

The purpose of this paper is to trigger the process for such investigations and attention. Precisely, we establish some analytical properties such as monotonicity and inequalities involving the hyperbolic sine integral function. As applications of some of the established properties, we obtain some rational bounds for the hyperbolic tangent function. We present our findings in the subsequent sections.

\section{Some Properties of the Hyperbolic Sine Integral}

The hyperbolic sine integral function may also be defined for $z\in(-\infty,\infty)$ by the following equivalent forms.
\begin{align}
\mathrm{Shi}(z)&=\int_{0}^{1} \frac{\sinh(zt)}{t}\,dt , \label{eqn:HyperSin-Integral-IntRep-2}  \\
&= \sum_{r=0}^{\infty}\frac{z^{2r+1}}{(2r+1)(2r+1)!} . \label{eqn:HyperSin-Integral_SeriesRep} 
\end{align}
By change of variable, representation \eqref{eqn:HyperSin-Integral-IntRep-2} is obtained from \eqref{eqn:HyperSin-Integral-IntRep-1} and representation \eqref{eqn:HyperSin-Integral_SeriesRep} is obtained from either \eqref{eqn:HyperSin-Integral-IntRep-1} or  \eqref{eqn:HyperSin-Integral-IntRep-2} by using the series representation of $\frac{\sinh(z)}{z}$.

By utilizing representation \eqref{eqn:HyperSin-Integral-IntRep-2}, the derivatives of $\mathrm{Shi}(z)$ are obtained as follows.
\begin{equation}\label{eqn:Hyper-Sinc-Deri-Even}
\mathrm{Shi}^{(k)}(z)= \int_{0}^{1} t^{k-1}\sinh(zt)dt, \quad  k\in \{2m:m\in \mathbb{N}_0\} ,
\end{equation}
\begin{equation}\label{eqn:Hyper-Sinc-Deri-Even}
\mathrm{Shi}^{(k)}(z)= \int_{0}^{1} t^{k-1}\cosh(zt)dt, \quad  k\in \{2m+1:m\in \mathbb{N}_0\} ,
\end{equation}
where $\mathbb{N}_0=\{0,1,2,3,...\}$. In particular, the first and second derivatives are 
\begin{equation}\label{eqn:Hyper-Sinc-1st-Deri}
\mathrm{Shi}'(z)= \int_{0}^{1} \cosh(zt)dt=\frac{\sinh(z)}{z},
\end{equation}
\begin{equation}\label{eqn:Hyper-Sinc-2nd-Deri}
\mathrm{Shi}''(z)= \int_{0}^{1} t \sinh(zt)dt=\frac{\cosh(z)}{z} - \frac{\sinh(z)}{z^2}.
\end{equation}

\begin{remark}
Identity \eqref{eqn:Hyper-Sinc-2nd-Deri} implies that
\begin{equation}\label{eqn:Known-Ineq}
\cosh(z)>\frac{\sinh(z)}{z}
\end{equation}
for $z>0$ and this is well known in the literature.
\end{remark}

\begin{lemma}\label{lem:Ravi-Laxmi}
If a function $\frac{p(x)}{x}$ is increasing or decreasing on an interval $I$, then $p(x)$ supperadditive or subadditive on $I$ respectively.
\end{lemma}

\begin{proof}
See Lemma 3.2 of \cite{Ravi-Laxmi-2018-IJAM} or Theorem 3.1 of \cite{Arpad-Neuman-2005-JIPAM}.
\end{proof}

\begin{theorem}\label{thm:SupperAdd-Shi}
The function $\mathrm{Shi}(z)$ is supperadditive on $(0,\infty)$. That is, the inequality
\begin{equation}\label{eqn:SupperAdd-Shi}
 \mathrm{Shi}(u+v) > \mathrm{Shi}(u) + \mathrm{Shi}(v)
\end{equation}
holds for $u>0$ and $v>0$.
\end{theorem}

\begin{proof}[First Proof]
Let $A(z)=\frac{\mathrm{Shi}(z)}{z}$ for $z>0$. Then
\begin{align*}
z^2A'(z)&=z \mathrm{Shi}'(z) - \mathrm{Shi}(z) \\
&=\sum_{r=0}^{\infty}\frac{z^{2r+1}}{(2r+1)!}  - \sum_{r=0}^{\infty}\frac{z^{2r+1}}{(2r+1)(2r+1)!} \\
&=\sum_{r=0}^{\infty}\left[1-\frac{1}{2r+1} \right] \frac{z^{2r+1}}{(2r+1)!} \\
&>0.
\end{align*}
Hence $A(z)$ is increasing and the conclusion follows from Lemma \ref{lem:Ravi-Laxmi}.
\end{proof}

\begin{proof}[Second Proof]
Let $u>0$ and $v>0$. Then
\begin{align*}
\mathrm{Shi}(u+v)&=\int_{0}^{1} \frac{\sinh(ut+vt)}{t}\,dt  \\
&=\int_{0}^{1} \frac{\sinh(ut) \cosh(vt)}{t}\,dt +\int_{0}^{1} \frac{\cosh(ut) \sinh(vt)}{t}\,dt \\
&>\int_{0}^{1} \frac{\sinh(ut)}{t}\,dt +\int_{0}^{1} \frac{\sinh(vt)}{t}\,dt \\
&=\mathrm{Shi}(u) + \mathrm{Shi}(v)
\end{align*}
since $\cosh(z)>1$ for all $z\neq0$.
\end{proof}

\begin{proof}[Third Proof]
Let $\phi(u,v)=\mathrm{Shi}(u+v) - \mathrm{Shi}(u) - \mathrm{Shi}(v)$ for $u>0$ and $v>0$. Without loss of generality, let $v$ be fixed. Then
\begin{align*}
\frac{\partial}{\partial u}\phi(u,v)&=\mathrm{Shi}'(u+v) - \mathrm{Shi}'(u) \\
&=\int_{0}^{1} \cosh(ut+vt)dt - \int_{0}^{1} \cosh(ut)dt \\
&=\int_{0}^{1} \left[ \cosh(ut)\cosh(vt)+\sinh(ut)\sinh(vt) \right]\,dt -\int_{0}^{1} \cosh(ut)\,dt \\
&=\int_{0}^{1} \cosh(ut)[\cosh(vt)-1]dt + \int_{0}^{1} \sinh(ut)\sinh(vt)dt \\
&>0
\end{align*}
since $\cosh(z)>1$ for all $z\neq0$. Thus, $\phi(u,v)$ is increasing and so
\begin{equation*}
\phi(u,v)>\lim_{u\to0}\phi(u,v)=0
\end{equation*}
which gives the desired result.
\end{proof}

\begin{theorem}\label{thm:Shi-Ineqs}
The inequality
\begin{equation}\label{eqn:Shi-Ineq-Sum}
\mathrm{Shi}(u) + \mathrm{Shi}(v) > u+v
\end{equation}
holds for $u>0$ and $v>0$, and the inequality
\begin{equation}\label{eqn:Shi-Ineq-Ratio}
\frac{\mathrm{Shi}(u)}{\mathrm{Shi}(v)} \leq \frac{u}{v}
\end{equation}
holds for $0<u\leq v$.
\end{theorem}

\begin{proof}
The monotonicity property of the function $\frac{\mathrm{Shi}(z)}{z}$ implies that, for $z>0$, we have
\begin{equation*}
\frac{\mathrm{Shi}(z)}{z}>\lim_{z\to 0^+} \frac{\mathrm{Shi}(z)}{z}=1.
\end{equation*}
That is,
\begin{equation*}
\mathrm{Shi}(z)>z .
\end{equation*}
Hence for $u>0$ and $v>0$, we have $\mathrm{Shi}(u)>u$ and $\mathrm{Shi}(v)>v$ which results to \eqref{eqn:Shi-Ineq-Sum}. Likewise, for $0<u\leq v$, we have
\begin{equation*}
\frac{\mathrm{Shi}(u)}{u} \leq \frac{\mathrm{Shi}(v)}{v}
\end{equation*}
which results to \eqref{eqn:Shi-Ineq-Ratio}.
\end{proof}

\begin{theorem}\label{thm:StarshapedType}
Let $z>0$ and $\lambda \in(0,1)$. Then the inequality
\begin{equation}\label{eqn:StarshapedType}
 \mathrm{Shi}(\lambda z) > \lambda \mathrm{Shi}(z) 
\end{equation}
holds. If $\lambda>1$, then the inequality is reversed.
\end{theorem}

\begin{proof}
Let $\alpha(z)=\mathrm{Shi}(\lambda z) - \lambda \mathrm{Shi}(z)$ for $z>0$ and $\lambda \in(0,1)$. Then 
\begin{align*}
\alpha'(z)&= \lambda \left[ \mathrm{Shi}'(\lambda z) -  \mathrm{Shi}'(z) \right] <0
\end{align*}
since $\mathrm{Shi}'(z)$ is increasing for $z>0$. Hence $\alpha(z)$ is decreasing and then, we have
\begin{equation*}
\alpha(z) > \lim_{z\to0^+}\alpha(z)=0
\end{equation*}
which gives \eqref{eqn:StarshapedType}.
\end{proof}

\begin{theorem}\label{thm:Sum-of-Shi}
For $z>0$, the inequality
\begin{equation}\label{eqn:Sum-of-Shi}
\mathrm{Shi}(z) + \mathrm{Shi}(1/z) \geq 2 \int_{0}^{1} \frac{\sinh(t)}{t}dt \approx 2.11450
\end{equation}
holds. Equality is attained if $z=1$.
\end{theorem}

\begin{proof}
The case for $z=1$ is easily seen. Because of this, let $P(z)=\mathrm{Shi}(z) + \mathrm{Shi}(1/z)$ for $z\in(0,1)\cup(1,\infty)$. Then
\begin{equation*}
P'(z)= \mathrm{Shi}'(z) - \frac{1}{z^2} \mathrm{Shi}'(1/z),
\end{equation*}
which means that
\begin{align*}
zP'(z)&= \sinh(z) - \sinh(1/z) :=E(z)
\end{align*}
Since $\sinh(z)$ is increasing, then $E(z)<0$ if $z\in(0,1)$ and $E(z)>0$ if $z\in(1,\infty)$. Thus, $P(z)$ is decreasing on $(0,1)$ and increasing on $(1,\infty)$. Therefore, on both intervals, we have
\begin{equation*}
P(z)>\lim_{z\to1}P(z)=2\mathrm{Shi}(1)=2\int_{0}^{1} \frac{\sinh(t)}{t}dt \approx 2.11450
\end{equation*}
completing the proof.
\end{proof}

\begin{lemma}[\cite{Pinelis-2002-JIPAM}]\label{lem:LMR}
Let $-\infty \leq u<v \leq \infty$ and $p$ and $q$ be continuous functions that are differentiable on $(u,v)$, with $p(u+)=q(u+)=0$ or $p(v-)=q(v-)=0$. Suppose that $q(z)$ and $q'(z)$ are nonzero for all $z\in(u,v)$. If $\frac{p'(z)}{q'(z)}$ is increasing (or decreasing) on $(u,v)$, then $\frac{p(x)}{q(x)}$ is also  increasing (or decreasing) on $(u,v)$.
\end{lemma}
In the literature, Lemma \ref{lem:LMR} is referred to as l'Hospital rule for monotonicy. It has become a remarkable tool in proving various  results in mathematical analysis.

\begin{lemma}\label{lem:Increasing-Funct-Sinh-Over-Shi}
For $z>0$, the function $T(z)=\frac{\sinh(z)}{\mathrm{Shi}(z)}$ is increasing.  
\end{lemma}

\begin{proof}
For $z\in(0,\infty)$,  we have
\begin{equation*}
T(z)=\frac{\sinh(z)}{\mathrm{Shi}(z)} = \frac{p_1(z)}{q_1(z)} ,
\end{equation*}
where $p_1(z)=\sinh(z)$, $q_1(z)=\mathrm{Shi}(z)$ and $p_1(0)=q_1(0)=0$. Then
\begin{equation*}
\frac{p'_1(z)}{q'_1(z)}=\frac{z \cosh(z)}{\sinh(z)}= \frac{p_2(z)}{q_2(z)} 
\end{equation*}
where $p_2(z)=z \cosh(z)$, $q_2(z)=\sinh(z)$ and $p_2(0)=q_2(0)=0$. Then
\begin{align*}
\sinh^2(z) \left( \frac{p_2(z)}{q_2(z)} \right)'&= \left[ \cosh(z) +z\sinh(z) \right]\sinh(z)-z\cosh^2(z) \\
&=\cosh(z)\sinh(z)+z\left[\sinh^2(z)-\cosh^2(z) \right] \\
&=\cosh(z)\sinh(z)-z \\
&>0
\end{align*}
since $\cosh(z)>1$ and $\sinh(z)>z$ for $z>0$.
Thus, $\frac{p_1'(z)}{q_1'(z)}$ is increasing. Hence by Lemma \ref{lem:LMR}, the function $\frac{p_1(z)}{q_1(z)}$ is also  increasing. This completes the proof.
\end{proof}

\begin{theorem}\label{thm:Product-of-Shi}
For $z>0$, the inequality
\begin{equation}\label{eqn:Product-of-Shi}
\mathrm{Shi}(z) \mathrm{Shi}(1/z) \geq \left(\int_{0}^{1} \frac{\sinh(t)}{t}dt \right)^2 \approx 1.11778
\end{equation}
holds. Equality is attained if $z=1$.
\end{theorem}

\begin{proof}
The case for $z=1$ is easily seen. And so, let $Q(z)=\mathrm{Shi}(z)\mathrm{Shi}(1/z)$ and $\theta(z)=\ln Q(z)$ for $z\in(0,1)\cup(1,\infty)$. Then
\begin{align*}
z\theta'(z)&= z \frac{\mathrm{Shi}'(z)}{\mathrm{Shi}(z)} - \frac{1}{z^2}\frac{\mathrm{Shi}'(1/z)}{\mathrm{Shi}(1/z)}\\
&=\frac{\sinh(z)}{\mathrm{Shi}(z)} - \frac{\sinh(1/z)}{\mathrm{Shi}(1/z)} \\
&:=H(z).
\end{align*}
Because of Lemma \ref{lem:Increasing-Funct-Sinh-Over-Shi}, then $H(z)<0$ if $z\in(0,1)$ and $H(z)>0$ if $z\in(1,\infty)$. Subsequently, $Q(z)$ is decreasing on $(0,1)$ and increasing on $(1,\infty)$. Therefore, on both intervals, we have
\begin{equation*}
Q(z)>\lim_{z\to1}Q(z)=(\mathrm{Shi}(1))^2=\left(\int_{0}^{1} \frac{\sinh(t)}{t}dt \right)^2 \approx 1.11778
\end{equation*}
completing the proof.
\end{proof}

\begin{lemma}\label{lem:Decreasing-Funct-Sinh-Over-Shi2}
For $z>0$, the function $V(z)=\mathrm{Shi}(z) - \sinh(z)$ is decreasing and the inequality 
\begin{equation}\label{eqn:Ineq-Shi-Minus-Sinh}
 \mathrm{Shi}(z)  - \sinh(z) < 0
\end{equation}
holds.
\end{lemma}

\begin{proof}
We have
\begin{align*}
V'(z)&=\mathrm{Shi}'(z) -\cosh(z) \\
&=\frac{\sinh(z)}{z} - \cosh(z) <0
\end{align*}
as a result of \eqref{eqn:Known-Ineq}. Hence
\begin{equation*}
V(z)<\lim_{z\to0}V(z)=0
\end{equation*}
which proves \eqref{eqn:Ineq-Shi-Minus-Sinh}.
\end{proof}

\begin{lemma}\label{lem:Decreasing-Funct-Sinh-Over-Shi2}
For $z>0$, the function $K(z)=\frac{\sinh(z)}{\mathrm{Shi}^2(z)}$ is decreasing.  
\end{lemma}

\begin{proof}
For $z\in(0,\infty)$,  we have
\begin{equation*}
K(z)=\frac{\sinh(z)}{\mathrm{Shi}^2(z)} = \frac{p_1(z)}{q_1(z)} ,
\end{equation*}
where $p_1(z)=\sinh(z)$, $q_1(z)=\mathrm{Shi}^2(z)$ and $p_1(0)=q_1(0)=0$. Then
\begin{equation*}
\frac{p'_1(z)}{q'_1(z)}=\frac{z \cosh(z)}{2\mathrm{Shi}(z)\sinh(z)}= \frac{p_2(z)}{q_2(z)} 
\end{equation*}
where $p_2(z)=z \cosh(z)$, $q_2(z)=2\mathrm{Shi}(z)\sinh(z)$ and $p_2(0)=q_2(0)=0$. Then
\begin{align*}
2\mathrm{Shi}^2(z) \left( \frac{p_2(z)}{q_2(z)} \right)'&= \mathrm{Shi}(z)\coth(z) - z \mathrm{Shi}(z)\mathrm{cosech}^2(z) - \cosh(z) \\
&=\cosh(z) \left[ \frac{\mathrm{Shi}(z)}{\sinh(z)} - 1 \right] -  z \mathrm{Shi}(z)\mathrm{cosech}^2(z) \\
&<0
\end{align*}
as a result of \eqref{eqn:Ineq-Shi-Minus-Sinh}.
Thus, $\frac{p_1'(z)}{q_1'(z)}$ is decreasing. Hence by Lemma \ref{lem:LMR}, the function $\frac{p_1(z)}{q_1(z)}$ is also decreasing. This completes the proof.
\end{proof}

\begin{remark}
The increasing property of the function $\frac{\sinh(z)}{\mathrm{Shi}(z)}$ is equivalent to
\begin{equation}\label{eqn:New-Ineq-1}
z \mathrm{Shi}(z) \cosh(z) - \sinh^2(z) > 0.
\end{equation}
Also, the decreasing property of the function $\frac{\sinh(z)}{\mathrm{Shi}^2(z)}$ is equivalent to
\begin{equation}\label{eqn:New-Ineq-2}
z \mathrm{Shi}(z) \cosh(z) - 2\sinh^2(z) < 0.
\end{equation}
Combining \eqref{eqn:New-Ineq-1} and \eqref{eqn:New-Ineq-2} yields
\begin{equation}\label{eqn:New-Ineq-3}
\sinh^2(z)  < z \mathrm{Shi}(z) \cosh(z) < 2\sinh^2(z) 
\end{equation}
which is also equivalent to
\begin{equation}\label{eqn:New-Ineq-4}
\frac{\tanh(z)}{z}  < \frac{\mathrm{Shi}(z)}{\sinh(z)} < 2\frac{\tanh(z)}{z} .
\end{equation}
\end{remark}

\begin{theorem}\label{thm:Bounds-for-Shi}
For $z>0$, the inequality
\begin{equation}\label{eqn:Bounds-for-Shi}
\frac{z}{2} + \frac{\cosh(z)-1}{z} < \mathrm{Shi}(z) < 2\left(\frac{\cosh(z)-1}{z}\right)
\end{equation}
holds.
\end{theorem}

\begin{proof}
Recall that $t<\mathrm{Shi}(t)<\sinh(t)$ for $t>0$. Then, integrating over the interval $(0,z)$, we have 
\begin{equation*}
\int_{0}^{z}tdt<\int_{0}^{z}\mathrm{Shi}(t)dt<\int_{0}^{z}\sinh(t)dt
\end{equation*}
which gives
\begin{equation*}
\frac{z^2}{2} < z \mathrm{Shi}(z) -\cosh(z) +1 <\cosh(z)-1
\end{equation*}
and this simplifies to \eqref{eqn:Bounds-for-Shi}.
\end{proof}

\begin{theorem}\label{thm:HMI-Shi}
For $z>0$, the inequality
\begin{equation}\label{eqn:HMI-Shi}
\frac{2\mathrm{Shi}(z) \mathrm{Shi}(1/z)}{\mathrm{Shi}(z) + \mathrm{Shi}(1/z)} \leq \int_{0}^{1} \frac{\sinh(t)}{t}dt \approx 1.05725 
\end{equation}
holds. Equality is attained if $z=1$.
\end{theorem}

\begin{proof}
The case for $z=1$ is easily seen. On that note, let $\Psi(z)=\frac{2\mathrm{Shi}(z) \mathrm{Shi}(1/z)}{\mathrm{Shi}(z) + \mathrm{Shi}(1/z)}$ and $h(z)=\ln \Psi(z)$ for $z\in(0,1)\cup(1,\infty)$. Then
\begin{equation*}
h'(z)=\frac{\mathrm{Shi}'(z)}{\mathrm{Shi}(z)} - \frac{1}{z^2}\frac{\mathrm{Shi}'(1/z)}{\mathrm{Shi}(1/z)} - \frac{\mathrm{Shi}'(z)-\frac{1}{z^2}\mathrm{Shi}'(1/z)}{\mathrm{Shi}(z) + \mathrm{Shi}(1/z)}
\end{equation*}
which implies that
\begin{equation*}
z\left[ \mathrm{Shi}(z) + \mathrm{Shi}(1/z) \right]h'(z)=z\frac{\mathrm{Shi}'(z)}{\mathrm{Shi}(z)}\mathrm{Shi}(1/z)  - \frac{1}{z}\frac{\mathrm{Shi}'(1/z)}{\mathrm{Shi}(1/z)} \mathrm{Shi}(z).
\end{equation*}
This further gives rise to
\begin{align*}
z\left[ \frac{1}{\mathrm{Shi}(z) }+ \frac{1}{\mathrm{Shi}(1/z)} \right]h'(z)
&=z\frac{\mathrm{Shi}'(z)}{\mathrm{Shi}^2(z)} - \frac{1}{z}\frac{\mathrm{Shi}'(1/z)}{\mathrm{Shi}^2(1/z)} \\
&=\frac{\sinh(z)}{\mathrm{Shi}^2(z)} - \frac{\sinh(1/z)}{\mathrm{Shi}^2(1/z)} \\
&=D(z).
\end{align*}
Owing to Lemma \ref{lem:Decreasing-Funct-Sinh-Over-Shi2}, we have $D(z)>0$ if $z\in(0,1)$ and $D(z)<0$ if $z\in(1,\infty)$. Thus, $h(z)$ is increasing on $(0,1)$ and decreasing on $(1,\infty)$. Accordingly, $\Psi(z)$ is increasing on $(0,1)$ and decreasing on $(1,\infty)$. Therefore, on both intervals, we have
\begin{equation*}
\Psi(z)<\lim_{z\to1}\Psi(z)=\mathrm{Shi}(1)=\int_{0}^{1} \frac{\sinh(t)}{t}dt \approx 1.05725 
\end{equation*}
completing the proof.
\end{proof}

\begin{remark}
Theorem \ref{thm:HMI-Shi} can be interpreted to mean that, for $z>0$, the harmonic mean of $\mathrm{Shi}(z)$ and $\mathrm{Shi}(1/z)$ can never be greater than the quantity $\mathrm{Shi}(1)$. Inequality \eqref{eqn:HMI-Shi} can also be rearranged as 
\begin{equation}
\frac{1}{2}\left[ \frac{1}{\mathrm{Shi}(z) }+ \frac{1}{\mathrm{Shi}(1/z)} \right] \geq \left( \int_{0}^{1} \frac{\sinh(t)}{t}dt \right)^{-1}.
\end{equation}
\end{remark}

\begin{lemma}[\cite{Niculescu-2000-MIA}]\label{lem:Geo-Convexity-Rel}
Let the function $\alpha:I\subseteq (0,\infty) \rightarrow  (0,\infty)$ be differentiable. Then $\alpha(z)$ is is geometrically convex (concave) if and only if $\frac{z\alpha'(z)}{\alpha(z)}$ is increasing (decreasing) respectively.
\end{lemma}

\begin{theorem}\label{thm:GeoConvexity-Shi}
The function $\mathrm{Shi}(z)$ is geometrically convex on $(0,\infty)$. That is, the inequality
\begin{equation}\label{eqn:GeoConvexity-Shi}
 \mathrm{Shi}(u^{k}v^{1-k})\leq \left(\mathrm{Shi}(u) \right)^{k} \left(\mathrm{Shi}(v) \right)^{1-k}
\end{equation}
holds for $u>0$, $v>0$ and $k\in[0,1]$.
\end{theorem}

\begin{proof}
Applying Lemma \ref{lem:Increasing-Funct-Sinh-Over-Shi}, we have 
\begin{equation*}
\frac{d}{dz}\left(\frac{z \mathrm{Shi}'(z)}{\mathrm{Shi}(z)}\right)=\frac{d}{dz}\left(\frac{\sinh(z)}{\mathrm{Shi}(z)}\right)>0
\end{equation*}
and by Lemma \ref{lem:Geo-Convexity-Rel}, we conclude that $\mathrm{Shi}(z)$ is geometrically convex. This is equivalent to \eqref{eqn:GeoConvexity-Shi}.
\end{proof}

\begin{remark}
It is interesting to note that, by letting $u=z$, $v=1/z$ and $k=\frac{1}{2}$ in \eqref{eqn:GeoConvexity-Shi}, we recover the inequality \eqref{eqn:Product-of-Shi}.
\end{remark}

\section{Rational Bounds for the Hyperbolic Tangent Function}

In this section, as applications of the hyperbolic sine integral, we obtain some rational bounds for the hyperbolic tangent function.

\begin{theorem}\label{thm:Rational-Bounds-1}
For $z>0$, the inequalities
\begin{equation}\label{eqn:Rational-Bounds-1}
\frac{2z}{z^2+2} <\tanh(z) < \frac{z^3+6z}{3z^2+6}
\end{equation}
hold. 
\end{theorem}

\begin{proof}
By direct computations, we obtain
\begin{align*}
\mathrm{Shi}^{(3)}(z)&= \int_{0}^{1} t^{2}\cosh(zt)dt  \\
&= \frac{(z^2+2)\sinh(z)-2z \cosh(z)}{z^3}>0.
\end{align*}
Upon rearrangement, we obtain
\begin{equation*}
\tanh(z)>\frac{2z}{z^2+2}
\end{equation*}
which gives the left hand side of \eqref{eqn:Rational-Bounds-1}. Also, 
\begin{align*}
\mathrm{Shi}^{(4)}(z)&= \int_{0}^{1} t^{3}\sinh(zt)dt  \\
&= \frac{(z^3+6z)\cosh(z)-(3z^2+6) \sinh(z)}{z^4}>0.
\end{align*}
Hence
\begin{equation*}
\tanh(z)<\frac{z^3+6z}{3z^2+6}
\end{equation*}
which gives the right hand side of \eqref{eqn:Rational-Bounds-1}. This completes the proof.
\end{proof}

\begin{theorem}\label{thm:Rational-Bounds-2}
For $z>0$, the inequalities
\begin{equation}\label{eqn:Rational-Bounds-2}
\frac{4z^3+24z}{z^4+12z^2+24} <\tanh(z) < \frac{z^5+20z^3+120z}{5z^4+60z^2+120}
\end{equation}
hold. 
\end{theorem}

\begin{proof}
By direct computations, we obtain
\begin{align*}
\mathrm{Shi}^{(5)}(z)&= \int_{0}^{1} t^{4}\cosh(zt)dt  \\
&= \frac{(z^4+12z^2+24)\sinh(z)- (4z^3+24z) \cosh(z)}{z^5}>0.
\end{align*}
This implies that
\begin{equation*}
\tanh(z)> \frac{4z^3+24z}{z^4+12z^2+24}
\end{equation*}
which gives the left hand side of \eqref{eqn:Rational-Bounds-2}. Also, 
\begin{align*}
\mathrm{Shi}^{(6)}(z)&= \int_{0}^{1} t^{5}\sinh(zt)dt  \\
&= \frac{(z^5+20z^3+120z)\cosh(z)-(5z^4+60z^2+120) \sinh(z)}{z^6} \\
&>0.
\end{align*}
Hence
\begin{equation*}
\tanh(z) < \frac{z^5+20z^3+120z}{5z^4+60z^2+120}
\end{equation*}
which gives the right hand side of \eqref{eqn:Rational-Bounds-2}. This completes the proof.
\end{proof}

\begin{theorem}\label{thm:Rational-Bounds-3}
For $z>0$, the inequalities
\begin{equation}\label{eqn:Rational-Bounds-3}
\frac{6z^5+120z^3+720z}{z^6+30z^4+360z^2+720} <\tanh(z) < \frac{z^7+42z^5+840z^3+5040z}{7z^6+210z^4+2520z^2+5040}
\end{equation}
hold. 
\end{theorem}

\begin{proof}
By direct computations, we obtain
\begin{align*}
\mathrm{Shi}^{(7)}(z)&= \int_{0}^{1} t^{6}\cosh(zt)dt  \\
&= \frac{(z^6+30z^4+360z^2+720)\sinh(z)- (6z^5+120z^3+720z) \cosh(z)}{z^7}\\
&>0.
\end{align*}
This implies that
\begin{equation*}
\tanh(z)> \frac{6z^5+120z^3+720z}{z^6+30z^4+360z^2+720} 
\end{equation*}
which gives the left hand side of \eqref{eqn:Rational-Bounds-3}. Also, 
\begin{align*}
&\mathrm{Shi}^{(8)}(z) \\
&= \int_{0}^{1} t^{7}\sinh(zt)dt  \\
&= \frac{(z^7+42z^5+840z^3+5040z)\cosh(z)-(7z^6+210z^4+2520z^2+5040) \sinh(z)}{z^8} \\
&>0.
\end{align*}
Hence
\begin{equation*}
\tanh(z) < \frac{z^7+42z^5+840z^3+5040z}{7z^6+210z^4+2520z^2+5040}
\end{equation*}
which gives the right hand side of \eqref{eqn:Rational-Bounds-3}. This completes the proof.
\end{proof}

\begin{theorem}\label{thm:Rational-Bounds-4}
For $z>0$, the inequalities
\begin{multline}\label{eqn:Rational-Bounds-4}
\frac{8z^7+336z^5+6720z^3+40320z}{z^8+56z^6+1680z^4+20160z^2+40320} <\tanh(z) \\
< \frac{z^9+72z^7+3024z^5+60480z^3+362880z}{9z^8+504z^6+15120z^4+181440z^2+362880}
\end{multline}
hold. 
\end{theorem}

\begin{proof}
By direct computations, we obtain
\begin{align*}
\mathrm{Shi}^{(9)}(z)&= \int_{0}^{1} t^{8}\cosh(zt)dt  \\
&= \frac{1}{z^9}\left[ (z^8+56z^6+1680z^4+20160z^2+40320)\sinh(z) \right. \\
& \quad \left. - (8z^7+336z^5+6720z^3+40320z) \cosh(z) \right] \\
&>0.
\end{align*}
This implies that
\begin{equation*}
\tanh(z)> \frac{8z^7+336z^5+6720z^3+40320z}{z^8+56z^6+1680z^4+20160z^2+40320}
\end{equation*}
which gives the left hand side of \eqref{eqn:Rational-Bounds-4}. Also, 
\begin{align*}
\mathrm{Shi}^{(10)}(z)&= \int_{0}^{1} t^{9}\sinh(zt)dt  \\
&= \frac{1}{z^{10}}\left[ (z^9+72z^7+3024z^5+60480z^3362880z)\cosh(z) \right. \\
& \quad \left. - (9z^8+504z^6+15120z^4+181440z^2+362880) \sinh(z) \right] \\
&>0.
\end{align*}
Hence
\begin{equation*}
\tanh(z) < \frac{z^9+72z^7+3024z^5+60480z^3+362880z}{9z^8+504z^6+15120z^4+181440z^2+362880}
\end{equation*}
which gives the right hand side of \eqref{eqn:Rational-Bounds-4}. This completes the proof.
\end{proof}

\begin{remark}
The bounds in \eqref{eqn:Rational-Bounds-4} are better than those in \eqref{eqn:Rational-Bounds-3}. The bounds in \eqref{eqn:Rational-Bounds-3} are also better than those in \eqref{eqn:Rational-Bounds-2}. And the bounds in \eqref{eqn:Rational-Bounds-2} are also better than those in \eqref{eqn:Rational-Bounds-1}. 
\end{remark}

\begin{remark}
Due to their monotonicity properties, for $m\ge2$, the derivatives of the hyperbolic sine integral, $\mathrm{Shi}^{(m)}(z)$ give rational bounds  for the hyperbolic tangent function. Particularly, odd derivatives give lower bounds and even derivatives give upper bounds. The corresponding bounds get better as $m$ increases. It is also observed that, the lower bounds obtained this way, are of the form $\frac{p'(z)}{p(z)}$ and the upper bounds are of the form $\frac{q(z)}{q'(z)}$ for some polynomials $p(z)$ and $q(z)$.
\end{remark}

As a byproduct of Theorem \ref{thm:Rational-Bounds-1}, we obtain the following result which  provides bounds for the hyperbolic cosine function.
\begin{corollary}\label{cor:Bounds-Cosh}
For $z>0$, the inequalities
\begin{equation}\label{eqn:Bounds-Cosh}
\frac{z^2+2}{2} <\cosh(z) < e^{\frac{z^2}{6}} \left( \frac{z^2+2}{2} \right)^{\frac{2}{3}}
\end{equation}
hold. 
\end{corollary}

\begin{proof}
By integrating \eqref{eqn:Rational-Bounds-1} over the interval $(0,z)$, we have
\begin{equation*}
\int_{0}^{z}\frac{2t}{t^2+2}dt <\int_{0}^{z} \tanh(t) dt < \int_{0}^{z} \frac{t^3+6t}{3t^2+6} dt
\end{equation*}
which gives
\begin{equation*}
\ln (z^2+2) - \ln 2 < \ln \cosh(z)  < \frac{z^2}{6} + \frac{2}{3} \ln (z^2+2) - \frac{2}{3}\ln2.
\end{equation*}
That is
\begin{equation*}
\ln \frac{z^2+2}{2}<\ln \cosh(z)  < \ln \left\{e^{\frac{z^2}{6}} \left( \frac{z^2+2}{2} \right)^{\frac{2}{3}} \right\}
\end{equation*}
and by taking exponents, we obtain \eqref{eqn:Bounds-Cosh}.
\end{proof}

\bibliographystyle{plain}

\begin{thebibliography}{99}

\bibitem{Abramowitz-Stegun-1965-DP} M. Abramowitz and I. A. Stegun, \textit{Handbook of Mathematical Functions with
formulas, Graphic and Mathematical Tables}, Dover Publications, Inc., New York, (1965).

\bibitem{Arpad-Neuman-2005-JIPAM} B. Arpad and E. Neuman \textit{Inequalities Involving Generalized Bessel Functions}, J. Inequal. Pure and Appl. Math., 6(4)(2005), Art No. 126, 1-9.

\bibitem{Bagul-2017-JMI} Y. J. Bagul \textit{Inequalities Involving Circular, Hyperbolic and Exponential Functions}, J. Math. Inequal., 11(3)(2017), 695-699.

\bibitem{Bagul-Chesneau-2019-IJOPCM} Y. J. Bagul and  C. Chesneau \textit{Two double sided inequalities involving sinc and hyperbolic sinc functions}, Int. J. Open Problems Compt. Math., 12(4)(2019), 15-20.

\bibitem{Bagul-Chesneau-2019-CUBO} Y. J. Bagul and  C. Chesneau \textit{Some New Simple Inequalities Involving Exponential, Trigonometric and Hyperbolic Functions}, Cubo. A Mathematical Journal., 21(1)(2019), 21-35.

\bibitem{Bagul-Etal-2021-Axioms} Y. J. Bagul, R. M. Dhaigude, M. Kostic and C. Chesneau \textit{Polynomial-Exponential Bounds for Some Trigonometric and Hyperbolic Functions}, Axioms, 10 (2021), Art No. 308, 1-10.

\bibitem{Bagul-Etal-2022-AMS} Y. J. Bagul, S. B. Thool , C. Chesneau and R. M. Dhaigude \textit{Refinements of Some Classical Inequalities Involving Sinc and Hyperbolic Sinc Functions}, Ann. Math. Sil., (2022), Published Ahead of Print. DOI: https://doi.org/10.2478/amsil-2022-0019.

\bibitem{Li-Miao-Guo-2022-Axioms} W-H. Li, P. Miao and B-N. Guo \textit{Bounds for the Neuman-Sandor Mean in Terms of the Arithmetic and Contra-Harmonic Means}, Axioms, 11(5)(2022), Art No. 236, 1-12.

\bibitem{Merca-2016-JNT} M. Merca \textit{The cardinal sine function and the Chebyshev-Stirling numbers}, J. Number Theory, 160(2016), 19-31.

\bibitem{Nantomah-Prempeh-2020-MJPAA} K. Nantomah and E. Prempeh \textit{Some Inequalities for Generalized Hyperbolic Functions}, Moroccan J. of Pure and Appl. Anal., 6(1)(2020), 76-92. 


\bibitem{Niculescu-2000-MIA} C. P. Niculescu \textit{Convexity according to the geometric mean}, Math. Inequal. Appl., 2(2)(2000), 155-167.

\bibitem{Nijimbere-2018-UMJ} V. Nijimbere \textit{Evaluation of Some Non-Elementary Integrals Involving Sine, Cosine, Exponential and Logarithmic Integrals: Part II}, Ural Math. J., 4(1)(2018), 43-55.

\bibitem{Pinelis-2002-JIPAM} I. Pinelis \textit{L'hospital type Rules for Monotonicity, with Applications}, J. Inequal. Pure Appl. Math., 3(1)(2002), Art No. 5, 5 pages.

\bibitem{Ravi-Laxmi-2018-IJAM} B. Ravi and A. V. Laxmi \textit{Subadditive and completely monotonic properties of the tricomi confluent hypergeometric functions}, Inter. J. Adv. Math., 5(2018), 25-33.

\bibitem{Sanchez-2012-TCMJ} J. S\'{a}nchez-Reyes \textit{The Hyperbolic Sine Cardinal and the Catenary}, College Math. J., 43(4)(2012), 285-290.

\bibitem{Sandor-2017-NNTDM} J. Sandor \textit{Two applications of the Hadamard integral inequality}, Notes Number Theory Discrete Math., 23(4)(2017), 52-55.

\bibitem{Zhu-2010-JIA} L. Zhu \textit{Inequalities for Hyperbolic Functionsand Their Applications}, J. Inequal. Appl., 2010(2010), Art No. 130821, 10 pages.







\end{thebibliography}


\end{document}